\documentclass[reqno,10pt]{article}
\usepackage{graphicx}
\usepackage{amsmath}
\usepackage{amssymb}
\usepackage{amsthm}
\usepackage{enumitem}
\usepackage{bbm}
\usepackage[l2tabu,orthodox]{nag}
\usepackage[all,error]{onlyamsmath}
\usepackage[strict]{csquotes}
\usepackage{url}
\def\N {\mathbb{N}}

\def\H {\mathcal{H}}

\def\C {\mathbb{C}}
\def\R {\mathbb{R}}

\def\E {\mathcal{E}}
\def\K{\mathcal{K}}

\def\Her{\mathbb{H}}
\newtheorem{theorem}{Theorem}
\newtheorem{lemma}[theorem]{Lemma}
\newtheorem{corollary}[theorem]{Corollary}
\newtheorem{prop}[theorem]{Proposition}
\theoremstyle{definition}
\newtheorem{definition}[theorem]{Definition}
\newtheorem{remark}[theorem]{Remark}

\numberwithin{equation}{section}
\numberwithin{theorem}{section}

\newenvironment{OMabstract}{\noindent\textbf{Abstract.} }{\medskip}
\newenvironment{OMsubjclass}{\noindent\textbf{Mathematics Subject Classification (2020):} }{\medskip}
\newenvironment{OMkeywords}{\noindent\textbf{Keywords:}  }{\medskip}

\begin{document}

\author{Mario Alberto Ruiz Caballero} %Please write full names and surnames of all co-authors here.
\title{On spectral stability for self-adjoint extensions}
\maketitle

\begin{OMabstract}
    %------------------------- Please type your abstract here ------------------
    We prove that given a symmetric completely non-selfadjoint operator $B$ with finite deficiency indices $(n,n)$ and a boundary triplet 
 $\left( \C^{n},\Gamma_{1},\Gamma_{2}\right)$ for $B^{*}$, the set of points in the spectrum of $A_{1}$ (the
 self-adjoint extension with domain $Ker\;\Gamma_{1}$) which are not eigenvalues of maximum multiplicity for any self-adjoint extension of $B$ disjoint of $A_{1}$, is a dense $\textit{G}_{\delta}$ set in $\sigma(A_{1})$. Furthermore, a proof of a Malamud's theorem that generalizes a well-known result of the Aronszajn-Donoghue theory on the characterization of eigenvalues is offered.
    %---------------------------------------------------------------------------
\end{OMabstract}

\begin{OMkeywords}
    %------------------------- Please type your keywords here ------------------
    boundary triplet, self-adjoint extension, completely non-selfadjoint operator, finite deficiency indices, eigenvalue of maximum multiplicity, Aronszajn-Donoghue theory.
    %---------------------------------------------------------------------------
\end{OMkeywords}

\begin{OMsubjclass}
    %------------------------- Please type your 2020 MSC numbers here ----------
    47B02, 47B25, 47A55, 47A10.
    %---------------------------------------------------------------------------
\end{OMsubjclass}

%------------------- Type contents of your article below --------------------
\section{Introduction}
In this work we continue the study started in \cite{Ruiz}. The results in such
paper, in particular Theorem 1.1 and Proposition 4.3, motivated the research on spectral stability for a family of
self-adjoint extensions of a symmetric operator with finite deficiency indices $(n,n)$ where $n>1$
(for the case of rank one regular perturbations and Sturm-Liouville operators, see \cite{Gordon, RMS, Rio}).
 Using Boundary Triplets Theory and denoting by $\Her_{n}$ the set of symmetric matrices $n\times n$ on $\C$, the main
 result is the following:
 
 %a theorem on forbidden energies-like is obtained for a sub-family of self-adjoint extensions of a simple operator with finite deficiency indices for eigenvalues of maximal multiplicity. 
 \begin{theorem}[Forbidden energies-like]\label{maintheorem}
Let $B$ denote a symmetric completely non-selfadjoint operator on a Hilbert space $\H$ with finite deficiency indices $(n,n)$ and
$\Pi=\left( \C^{n},\Gamma_{1},\Gamma_{2}\right)$ be a boundary triplet for $B^{*}$. Then 
\begin{equation}\label{u}
\left\lbrace x\in\sigma(A_{1}):x\not\in\sigma_{p}^{max}(A_{D}),\;for\;any\;D\in\Her_{n}\right\rbrace 
\end{equation}
is dense $G_{\delta}$ in $\sigma(A_{1})$.
\end{theorem}
 To prove it, a generalization of a theorem of the Aronszajn-Donoghue theory is required which was proposed by M. Malamud in \cite{MALAMUD}. It is worth mentioning that a proof of said theorem does not appear in \cite{MALAMUD}, so a proof is offered in this paper.\\
 
The paper is diveded as follows. In Section 2 we define boundary triplets, disjoint proper closed extensions
and state the extension theorem by boundary triplets. In Section 3 the definitions of matrix-valued measure, 
Nevanlinna-Herglotz matrix and integral with respect to matrix-valued measures through quadratic forms are given.
Some results of \cite{Ruiz} for the scalar case are extended to the matrix-valued one. In section 4 we remind results
of \cite{MALAMUD2, SOLBIR} into the approach of direct integrals of Hilbert spaces. We define the multiplicity functions of both matrix-valued measures and self-adjoint operators. We also define the space $L^{2}(\Omega, \C^{n})$,
where $\Omega$ is a matrix-valued measure, and state the Kats Theorem. In Section 5 we define Weyl functions, completely non-selfadjoint operators as well as
eigenvalues of maximum multiplicity. The aforementioned Malamud's theorem is proved  by generalizing the classical method of \cite{Aron} and employing Kats theory. In Section 6, we present a similar result to Malamud's theorem by
changing the rol of the operator $A_{1}$ by anoter self-adjoint extension. Moreover Theorem \ref{maintheorem} is obtained as an immediate consequence of Malamud's theorem.

 \section{Boundary Triplets Theory}
 
This section relies on \cite{MALAMUD} and \cite[Chapter 14]{Sch}.
%We begin with the Boundary Triplets Theory (see \cite{Sch, MALAMUD, MALAMUD2}).

\begin{definition}
Let $B$ denote a densely defined symmetric operator on  $\H$. We shall call deficiency spaces to the subspaces
%\footnote{Given an operator $A$ densely defined on a Hilbert space, we denote as $A^{*}$ its adjoint operator that acts on the same Hilbert space.}

\begin{equation*}
K_{\pm}(B) :=Ran(B\pm iI)^{\perp}=Ker(B^{*}\mp iI).
\end{equation*}
Also, we shall call deficiency indices to $B$ the pair $\left( d_{+}(B), d_{-}(B)\right) $ where
\begin{equation*}
d_{\pm}(B):=dim\;K_{\pm}(B).
\end{equation*}
\end{definition}

\begin{definition}
Let $B$ denote a densely defined symmetric operator on $\H$ with equal deficiency indices. A triplet
$\left( \K,\Gamma_{1},\Gamma_{2}\right)$ where $\K$ is a Hilbert space such that $dim\;\K=d_{\pm}(H)$ and
$\Gamma_{i}:D(B^{*})\longrightarrow\K$ is a linear mapping with $i=1,2$ is said to be a boundary triplet for $B^{*}$ if:
\begin{itemize}
\item Satisfies the abstract Green identity, i. e.,
\begin{equation*}
\langle B^{*}x,y\rangle -\langle x,B^{*}y\rangle=\langle \Gamma_{2}x,\Gamma_{1}y\rangle_{\K} -\langle \Gamma_{1}x,\Gamma_{2}y\rangle_{\K}
\end{equation*}
for each $x,y\in D(B^{*})$.
\item The linear mapping $\Gamma=\Gamma_{1}\times\Gamma_{2}:D(B^{*})\longrightarrow\K\bigoplus\K$ given by
$\Gamma x=\left( \Gamma_{1}x,\Gamma_{2}x\right) $ for each $x\in D(B^{*})$ is surjective.
\end{itemize}
We denote by $A_{i}$ with $i=1,2$ the self-adjoint extension of $B$ such that $D(A_{i})=Ker\;\Gamma_{i}$.
\end{definition}
\begin{definition}
Let $B$ be a densely defined symmetric operator on $\H$. We define a closed proper extension of $B$ as a closed operator $A$ on $\H$ such that $B\subseteq A\subseteq B^{*}$. We say that two proper closed extensions $A$ and $A'$ of $B$ are disjoint if $D(A)\cap D(A')=D(B)$. 
\end{definition}

We state the Extension Theorem by Boundary Triplets.

\begin{theorem}[Proposition 14.7, \cite{Sch} and Proposition 1, \cite{MALAMUD}]\label{BT-theorem}
Let $B$ denote a densely defined closed symmetric operator on $\H$ with equal deficiency indices.
We denote by $\mathcal{C}(B)$ the set of proper closed extensions of $B$. Then there exists a boundary triplet
$\left( \K,\Gamma_{1},\Gamma_{2}\right)$ for $B^{*}$ such that there is a bijection between $\mathcal{ C}(B)$ and the
set of closed relations on $\K$ as follows:
\begin{itemize}
\item If $\Lambda$ is a closed relation on $\K$, then $A_{\Lambda}\in\mathcal{C}(B)$ is given by
\begin{equation*}
D(A_{\Lambda}):=\left\lbrace x\in D(B^{*}):\Gamma x\in\Lambda \right\rbrace.
\end{equation*}
\item If $A\in\mathcal{C}(B)$, then $\Lambda_{A}:=\Gamma \left( D(A)\right)$ is a closed relation on $\K$.
\end{itemize}
In addition,
\begin{itemize}
\item If $\Lambda$ is a closed relation on $\K$, then $(A_{\Lambda})^{*}=A_{\Lambda^{*}}$.
\item If $A\in\mathcal{C}(B)$ is a symmetric operator, then $\Lambda_{A}$ is a closed symmetric relation on $\K$.
\item If $\Lambda$ is a closed symmetric relation on $\K$, then $A_{\Lambda}\in\mathcal{C}(B)$ is symmetric.
\item Let $A_{1}$ be the operator such that $D(A_{1})=Ker\;\Gamma_{1}$. Then $A_{\Lambda}$ is disjoint from $A_{1}$ if
and only if there is a densely defined closed operator $D$ on $\K$ such that $\Lambda=graph(D)$. If so, $A_{\Lambda}$ is denoted by $A_{D}$
and
\begin{equation*}
D(A_{D})=Ker\;(\Gamma_{2}-D\Gamma_{1}).
\end{equation*}
\end{itemize}
\end{theorem}
 
\section{Nevanlinna-Herglotz matrices and matrix-valued measures}

Let $\textit{B}(\R)$ be the Borel $\sigma$-algebra on $\R$ and $\C^{n\times n}$ the space of $n\times n$ matrices
on $\C$. See Section 5 of \cite{GT} for the following definition.

\begin{definition}
A matrix-valued measure $\Omega:\textit{B}(\R)\longrightarrow\C^{n\times n}$ is a function such that  
\begin{itemize}
\item[1)] For every $X\in\textit{B}(\R)$ bounded, $\Omega(X)$ is a positive-definite matrix.
\item[2)] $\Omega(\varnothing)=0$.
\item[3)] For $\left\lbrace X_{n}\right\rbrace_{n\in\N}\subset\textit{B}(\R) $ where $X_{m}\cap X_{l}=\varnothing$
for each $m\neq l$ and $\bigcup_{n\in\N}X_{n}$ is bounded 
\begin{equation*}
\Omega\left( \bigcup_{n\in\N}X_{n}\right) =\sum_{n\in\N}\Omega(X_{n}).
\end{equation*}
\end{itemize}
We define the trace of $\Omega$ as
\begin{equation}\label{medidasasoci}
Tr\Omega:=\sum_{i=1}^{n}\mu_{ii}\;where\;\mu_{ij}:=\langle e_{i}, \Omega(\cdot)e_{j}\rangle_{\C^{n}}
\end{equation}
where $\left\lbrace e_{i}\right\rbrace_{i=1}^{n}$ is the canonical basis of $\C^{n}$.
\end{definition}

We propose the following definition.
 
\begin{definition}
Let $f:\R\longrightarrow\C$ be a Borel measurable function. We define $\Phi_{f}:\C^{n}\bigoplus\C^{n}\longrightarrow\C$ by
\begin{equation*}
\Phi_{f}(a,b):=\int_{\R} f(y)d\langle a,\Omega(y)b\rangle_{\C^{n}}
\end{equation*}
where $D(\Phi_{f}):=\left\lbrace(a,b)\in\C^{n}\bigoplus\C^{n}:\int_{\R} f(y)d\langle a,\Omega(y)b\rangle_{\C^{n}}\;exists\;\right\rbrace $.
\end{definition}

The above mapping is clearly a sesquilinear form. The following lemma is elemental.

\begin{lemma}
If $f:\R\longrightarrow\R$ is a Borel measurable function, then $\Phi_{f}$ is a symmetric form on $\C^{n}$.
\end{lemma} 

We propose the following definition.

\begin{definition}
If $f:\R\longrightarrow\C$ is a Borel measurable function, we define the expression $\int_{\R} fd\Omega$ as the matrix associated with the form $\Phi_{f}$.
\end{definition}

\begin{remark}
\begin{itemize}
\item[1)] If $f$ is real-valued, then $\int_{\R} fd\Omega$ is a symmetric matrix.
\item[2)] If $f$ is positive-valued, then $\int_{\R} fd\Omega$ is a positive-definite matrix.
\item[3)] The notation $\int_{\R} fd\Omega\in\C^{n\times n}$ means that the matrix exists, that is, the form $\Phi_{f}$ is defined everywhere in $\C^{n}$.
\end{itemize}
\end{remark}

\begin{definition}
A holomorphic function $M:\C\setminus\R\longrightarrow\C^{n\times n}$ is said to be a Nevanlinna-Herglotz matrix if
\begin{center}
$Imz\cdot Im\;M(z)>0$ and $M(\overline{z})=M(z)^{*}$ for all $z\in\C\setminus\R$
\end{center}
where $Im\;M(z):=\frac{1}{2i}\left( M(z)-M(z)^{*}\right) $.
\end{definition}

The following proposition is trivial.

\begin{prop}\label{me}
If $\left\lbrace t_{n}\right\rbrace_{n\in\N} $ is a sequence of sesquilinear forms on a finite-dimensional Hilbert space $\H$ such that $t_{n}\longrightarrow_{n\rightarrow\infty} t$ 
where $t$ is a sesquilinear form on $\H$, then $A_{n}\longrightarrow A$
where $\left\lbrace A_{n}\right\rbrace_{n\in\N} $ and $A$ are the associated operators.
%Conversely, if $A_{n}\longrightarrow_{w} A$, then $t_{n}\longrightarrow t$ where $\left\lbrace t_{n}\right\rbrace_{n\in\N} $ and $t$ are the sesquilinear forms of $\left\lbrace A_{n}\right\rbrace_{n\in\N} $ and $A$ respectively.
\end{prop}

We obtain the following generalization of \cite[Proposition 3.1]{Ruiz} 

\begin{prop}\label{implicaciongrl}
Let $\Omega$ be a matrix-valued measure and $x\in\R$ such that
\begin{equation*}
\int_{\R} \frac{d\Omega(y)}{1+y^{2}},\;T(x)\in\C^{n\times n}
\end{equation*}
where $T(x):=\int_{\R} \frac{d\Omega(y)}{(x-y)^{2}}$. Then
\begin{equation*}
M(x+i0):=\lim_{\varepsilon\longrightarrow 0}M(x+i\varepsilon)\in\C^{n\times n}\;and\;M(x+i0)=M(x+i0)^{*}
\end{equation*}
where $M$ is the Nevanlinna-Herglotz matrix associated with $\Omega$ such that
\begin{equation*}
\lim_{\varepsilon\longrightarrow \infty}\frac{M(x+i\varepsilon)}{i\varepsilon}=0.
\end{equation*}
\end{prop}
\begin{proof}
By Canonical Representation Theorem of a Nevanlinna-Herglotz matrix (see \cite[Theorem 5.4(iv)]{GT})
\begin{equation}\label{funcionM}
M(z)=C+\int_{\R} \left( \frac{1}{y-z}-\frac{y}{1+y^{2}}\right) d\Omega(y)
\end{equation}
where $C\in\C^{n\times n}$ is a symmetric matrix. We are going to assume, without loss of generality, that $C=0$.
Let $\Phi$ be and $\Phi_{x}$ quadratic forms of matrices $\int_{\R} \frac{d\Omega(y)}{1+y^{2}}$
and $T(x)$ respectively. Therefore, for every $c\in\C^{n}$
\begin{equation*}
\int_{\R} \frac{1}{1+y^{2}}d\langle c,\Omega(y)c\rangle_{\C^{n}}=:\Phi\left[ c\right] <\infty.
\end{equation*}
So, $m_{c}(z):=\int_{\R} \left( \frac{1}{y-z}-\frac{y}{1+y^{2}}\right)d\langle c,\Omega(y)c\rangle_{\C^{n}}$ exists
for every $c\in\C^{n}$. That is, the quadratic form of the Nevanlinna-Herglotz matrix 
\begin{equation*}
M(z)=\int_{\R} \left( \frac{1}{y-z}-\frac{y}{1+y^{2}}\right) d\Omega(y)
\end{equation*}
which is given by $\Phi_{z}\left[ c\right] :=m_{c}(z)$ is defined everywhere in $\C^{n}$. Moreover, $\Phi_{x}$ is
defined everywhere in $\C^{n}$ since $T(x)$ is a matrix. This says that, for each $c\in\C^{n}$
\begin{equation*}
\int_{\R} \frac{1}{(x-y)^{2}}d\langle c,\Omega(y)c\rangle_{\C^{n}}=\Phi_{x}\left[ c\right] <\infty.
\end{equation*}
By \cite[Proposition 3.1]{Ruiz},
\begin{equation*}
\lim_{\varepsilon\longrightarrow 0}\Phi_{x+i\varepsilon}[c]=\lim_{\varepsilon\longrightarrow 0}m_{c}(x+i\varepsilon)
=m_{c}(x+i0)\in\R.
\end{equation*}
That is, 
\begin{equation*}
\Phi_{x+i\varepsilon}\longrightarrow_{\varepsilon\longrightarrow 0}\Phi_{x+i0}.
\end{equation*}
By Proposition \ref{me}, it turns out that
\begin{equation*}
\lim_{\varepsilon\longrightarrow 0}M(x+i\varepsilon)=\int_{\R} \left( \frac{1}{y-x}-\frac{y}{1+y^{2}}\right) d\Omega(y).
\end{equation*}
Furthermore, $M(x+i0)$ is a symmetric matrix since the function $g(y):=\frac{1}{y-x}-\frac{y}{1+y^{2}}$ is real-valued. Thus, we conclude the result.

\end{proof}

The proof of the following result follows the same approach as \cite[Theorem 2.1]{RMS}.

\begin{theorem}\label{teo2.1gral}
Suppose that $\Omega$ is a matrix-valued measure such that
\begin{equation*}
\int_{\R} \frac{d\Omega(y)}{1+y^{2}}\in\C^{n\times n}.
\end{equation*}
Then
\begin{equation}\label{set}
\left\lbrace x\in supp\;\Omega:T(x)\not\in \C^{n\times n}\right\rbrace 
\end{equation}
where $T(x):=\int_{\R} \frac{d\Omega(y)}{(x-y)^{2}}$, is dense $\textit{G}_{\delta}$ in $supp\;\Omega$.
\end{theorem}
\begin{proof}
Consider the Nevanlinna-Herglotz matrix $M$ given by
\begin{equation*}
M(z):=\int_{\R} \left( \frac{1}{y-z}-\frac{y}{1+y^{2}}\right) d\Omega(y).
\end{equation*}

We assert that

\begin{equation*}
Int_{supp\;\Omega}\left\lbrace x\in supp\;\Omega :T(x)\in\C^{n\times n}\right\rbrace =\varnothing.
\end{equation*} 

Suppose the above set contains an open interval $S$. This means that each $x\in S$ satisfies that
$T(x)\in\C^{n\times n}$. Since $\int_{\R} \frac{d\Omega(y)}{1+y^{2}}\in\C^{n\times n}$, by
Proposition \ref{implicaciongrl}, we obtain that all $x\in S$ satisfies that
\begin{equation*}
M(x+i0)=\lim_{\varepsilon\longrightarrow 0}M(x+i\varepsilon)\in\C^{n\times n}\;and\;M(x+i0)=M(x+i0)^{*}.
\end{equation*}
We note that for all $x\in S$ the following occurs:
\begin{itemize}
\item Since $Im\,M(x+i0)=0$, $rank\;Im\,M(x+i0)=0$. Thus, $S\cap\Omega_{ac}=\varnothing$.
\item Due to that both the function $f(z):=Im\;z$ and the trace operator are continuous, for each $x\in S$
\begin{equation*}
\lim_{\varepsilon\longrightarrow 0}Im\;Tr\;M(x+i\varepsilon)=Im\;Tr\;M(x+i0)=0.
\end{equation*}
The above is because $M(x+i 0)$ is symmetric, so its main diagonal is real. Therefore, $S\cap\Omega_{s}=\varnothing$.
\end{itemize}
By \cite[Theorem 6.1]{GT}, $\Omega(S)=\Omega_{ac}(S)+\Omega_{s}(S)=0$. Then $\R\setminus S$
is a closed support for $\Omega$, that is, $supp\;\Omega\subseteq\R\setminus S$ which is equivalent to
$supp\;\Omega\cap S=\varnothing$. Then, we conclude that the set (\ref{set}) is dense
in $supp\;\Omega$.\\

Now, we assert that the set (\ref{set}) is $\textit{G}_{\delta}$ in $supp\;\Omega$. We note that
\begin{align*}
T(x)\in\C^{n\times n} &\Longleftrightarrow  \forall\;c\in\C^{n},\;\int_{\R}\frac{1}{(x-y)^{2}}d\langle c,\Omega(y)c\rangle_{\C^{n}}<\infty \\
&\Longleftrightarrow  \forall\;i=1,...,n,\;\int_{\R}\frac{1}{(x-y)^{2}}d\langle e_{i},\Omega(y)e_{i}\rangle_{\C^{n}}<\infty.
\end{align*}
  
Therefore,
\begin{equation*}
\left\lbrace x\in\R:T(x)\not\in \C^{n\times n}\right\rbrace=\bigcup_{i=1}^{n}W_{i}
\end{equation*}
where 
\begin{equation*}
W_{i}:=\left\lbrace x\in\R:\int_{\R}\frac{1}{(x-y)^{2}}d\langle e_{i},\Omega(y)e_{i}\rangle_{\C^{n}}=\infty\right\rbrace.
\end{equation*}

%Since the measure $\langle c,\Omega(\cdot)c\rangle_{\C^{n}}$ satisfies \ref{grande} for every , by
%Lemma \ref{lemaG}
By \cite[Lemma 3.2]{Ruiz}, for all $c\in\C^{n}$
\begin{equation*}
\int_{\R}\frac{1}{(x-y)^{2}}d\langle c,\Omega(y)c\rangle_{\C^{n}}=\lim_{m\longrightarrow\infty}\int_{\R}\frac{1}{(x-y)^{2}+\frac{1}{m^{2}}}d\langle c,\Omega(y)c\rangle_{\C^{n}}.
\end{equation*}
Thus,
\begin{align*}
W_{i}&=\left\lbrace x\in\R:\forall\;k\in\N\;\exists\;m\in\N\;such\;that\;\int_{\R}\frac{1}{(x-y)^{2}+\frac{1}{m^{2}}}d\langle e_{i},\Omega(y)e_{i}\rangle_{\C^{n}}>k\right\rbrace \\
&=\bigcap_{k\in\N}\bigcup_{m\in\N}\left\lbrace x\in\R:\int_{\R}\frac{1}{(x-y)^{2}+\frac{1}{m^{2}}}d\langle e_{i},\Omega(y)e_{i}\rangle_{\C^{n}}>k\right\rbrace 
\end{align*}
where the set between brackets is open since the respective function is continuous by \cite[Lemma 3.2]{Ruiz}. So,
\begin{align*}
\bigcup_{i=1}^{n}W_{i}&=\bigcup_{i=1}^{n}\bigcap_{k\in\N}\bigcup_{m\in\N}\left\lbrace x\in\R:\int_{\R}\frac{1}{(x-y)^{2}+\frac{1}{m^{2}}}d\langle e_{i},\Omega(y)e_{i}\rangle_{\C^{n}}>k\right\rbrace \\
&=\bigcap_{k\in\N}\bigcup_{i=1}^{n}\bigcup_{m\in\N}\left\lbrace x\in\R:\int_{\R}\frac{1}{(x-y)^{2}+\frac{1}{m^{2}}}
d\langle e_{i},\Omega(y)e_{i}\rangle_{\C^{n}}>k\right\rbrace
\end{align*}
Finally we consider intersection with $supp\;\Omega$ and conclude that (\ref{set}) is $\textit{G}_{\delta}$ in $supp\;\Omega$.
\end{proof}

\section{Kats theorem and multiplicity function}
%-----------------------------------------------------------------------------------
We state the following results.

\begin{theorem}[Theorem 4.3(i), \cite{MALAMUD2} and Theorem 7.5.1, \cite{SOLBIR}]\label{teoespectral}
Let $A$ be a self-adjoint operator on $\H$ and $\rho$ a scalar Borel measure such that $\E_{A}$ is equivalent to $\rho$. Then 
$A$ is unitarily equivalent to the operator $Q$ on $\int_{\R}\oplus G(t)d\rho(t)$
such that
\begin{equation*}
D(Q):=\left\lbrace f\in \int_{\R}\oplus G(t)d\rho(t):Qf\in \int_{\R}\oplus G(t)d\rho(t)\right\rbrace 
\end{equation*}
and $Qf(t):=tf(t)$.
\end{theorem}

\begin{theorem}[Theorem 4.3(ii), \cite{MALAMUD2}]\label{teoespectralii}
Let $Q_{i}$ be the operator on $\int_{\R}\oplus G_{i}(t)d\rho_{i}(t)$ such that
\begin{equation*}
D(Q_{i}):=\left\lbrace f\in \int_{\R}\oplus G_{i}(t)d\rho_{i}(t):Q_{i}f\in \int_{\R}\oplus G_{i}(t)d\rho_{i}(t)\right\rbrace 
\end{equation*}
and $Q_{i}f(t):=tf(t)$ with $i=1,2$. Then $Q_{1}$ is unitarily equivalent to $Q_{2}$ if and
only if $\rho_{1}$ is equivalent to $\rho_{2}$ and $dim\;G_{1}(t)=dim\;G_{2}(t)$, $\rho_{i}$-a.e. on $\R$.
\end{theorem}

We give the definition of multiplicity function of a self-adjoint operator.

\begin{definition}
If $A$ holds hypothesis of Theorem \ref{teoespectral}, we define the multiplicity function of $A$
given by
\begin{equation}\label{funcionmult}
N_{A}(t):=dim\;G(t).
\end{equation}
\end{definition}

The following is a version of \cite[Corollary 4.4]{MALAMUD2} and \cite[Theorem 7.5.2]{SOLBIR}

\begin{theorem}\label{operunitequiv}
Let $A_{i}$ be a self-adjoint operator on the Hilbert space $\H_{i}$, with $i=1,2$ such that hypothesis of Theorem \ref{teoespectral} holds. If $A_{1}$ is unitarily equivalent to $A_{2}$, then 
\begin{center}
$\E_{A_{1}}$ is equivalent to $\E_{A_{2}}$ and $N_{A_{1}}(t)=N_{A_{2}}(t)$, $\E_{A_{i}}$-a.e. on $\R$.
\end{center}
\end{theorem}
\begin{proof}
Let $\rho_{i}$ be a scalar measure equivalent to $\E_{A_{i}}$, with $i=1,2$. By Theorem \ref{teoespectral},
$A_{i}$ is unitarily equivalent to the multiplication operator $Q_{i}$ by $t$
 on $\int_{\R}\oplus G_{i}(t)d\rho_{i}(t)$. Then $Q_{1}$ is unitarily equivalent to $Q_{2}$ and by
Theorem \ref{teoespectralii}, $\E_{A_{1}}$ is equivalent to $\E_{A_{2}}$ and by (\ref{funcionmult}),
$N_{A_{1}}(t)=N_{A_{2}}(t)$, $\E_{A_{i}}$-a.e. on $\R$.
\end{proof}

Now we have the definition of multiplicity function of a matrix-valued measure.
%--------------------------------------------------------------------------------------
\begin{definition}[Definition 4.5, \cite{MALAMUD2}]
Let $\Omega$ be a matrix-valued measure. We define the density matrix
\begin{equation*}
\Psi_{\Omega}(t):=\left( \frac{d\mu_{ij}}{dTr\Omega}(t)\right) _{i,j=1}^{n}
\end{equation*}
where $\mu_{ij}$ and $Tr\Omega$ are given by (\ref{medidasasoci}). We call multiplicity function of
$\Omega$ to the expression
\begin{equation*}
N_{\Omega}(t):=rank\Psi_{\Omega}(t).
\end{equation*}
\end{definition}

Next we provide the definition of the Hilbert space $L^{2}(\Omega,\C^{n})$. For more details see
\cite{MALAMUD2, Kats}.

\begin{definition}
We define $L^{2}(\Omega,\C^{n})$ as the quotient space $\dot{L}^{2}(\Omega,\C^{n})\diagup Ker\;p$
where $\dot{L}^{2}(\Omega,\C^{n})$ is the set of Borel measurable functions
$f:\R\longrightarrow\C^{n}$ such that 
\begin{equation*}
p(f):=\int_{\R}\langle d\Omega(t)f(t),f(t)\rangle_{\C^{n}}<\infty
\end{equation*}
with the semiscalar product
\begin{equation*}
\langle f,g\rangle_{L^{2}(\Omega,\C^{n})}:=\int_{\R}\langle d\Omega(t)f(t),g(t)\rangle_{\C^{n}}.
\end{equation*}
\end{definition}

Now, we define to the spaces with semiscalar product $\dot{G}(t):=\left( \C^{n}, \langle\cdot,\cdot\rangle_{\dot{G}(t)}\right)$ where
\begin{equation}\label{spaceGoft}
\langle f,g\rangle_{\dot{G}(t)}:=\langle \Psi_{\Omega}(t)f,g\rangle_{\C^{n}}.
\end{equation}

Finally, we state the Kats Theorem which gives a functional description for $L^{2}(\Omega,\C^{n})$.

\begin{theorem}[Theorem 2.11, \cite{MALAMUD2}]\label{teokats}
If $\Omega$ is a matrix-valued measure, then 
\begin{equation*}
L^{2}(\Omega,\C^{n})=\int_{\R}\oplus G(t)dTr\Omega(t)
\end{equation*}
where $G(t):=\dot{G}(t)\diagup Ker\;p_{t}$, with $p_{t}$ the seminorm associated to (\ref{spaceGoft}), and
\begin{equation*}
\langle f,g\rangle_{L^{2}(\Omega,\C^{n})}=\int_{\R}\langle \Psi_{\Omega}(t)f(t),g(t)\rangle_{\C^{n}}.
\end{equation*}
In addition, $N_{\Omega}(t)=dim\;G(t)\leq n$.
\end{theorem}

\section{Aronszajn-Donoghue theory}

Let us remind the following definitions for the particular case of finite deficiency indices. Denote by $\varrho$ the resolvent set.

\begin{definition}
Let $B$ denote a densely defined closed symmetric operator on $\H$ with finite deficiency indices $(n,n)$ and
$\Pi=\left( \C^{n},\Gamma_{1},\Gamma_{2}\right)$ a boundary triplet for $B^{*}$. We define the Weyl function corresponding
to $\Pi$ as the function $M:\varrho(A_{1})\longrightarrow\C^{n}$ such that
\begin{center}
$\Gamma_{2}f_{z}=M(z)\Gamma_{1}f_{z}$ where $f_{z}\in Ker\;(B^{*}-zI)$ and $z\in\varrho(A_{1})$.
\end{center}
In addition, we define the Weyl function of an operator $A_{D}$ with $D\in\Her_{n}$ as
\begin{equation*}
M_{D}(z):=\left[ D-M(z)\right] ^{-1}.
\end{equation*}
We say that a closed operator $B$ is completely non-selfadjoint if there is no subspace reducing $B$ such that the part of
$B$ in this subspace is self-adjoint. 
%We also say that an operator is simple if is symmetric and completely
non-selfadjoint.\\
%there exists not operators $B_{1}$ and $B_{2}$ on $\H$ such that $B=B_{1}\oplus B_{2}$ where $B_{1}\neq 0$ is self-adjoint and $B_{2}$ is symmetric.\\

Let $B$ denote a symmetric completely non-selfadjoint operator on $\H$ with equal and finite deficiency indices. Let $A$ be any of its self-adjoint extensions. We say that an eigenvalue $z$ of $A$ is of maximum multiplicity if
\begin{equation*}
dim\;Ker(A-zI)=d_{\pm}(B).
\end{equation*}
Denote by $\sigma_{p}^{max}(A)$ the set of eigenvalues of maximum multiplicity of $A$.
\end{definition}

%\begin{definition}
%Dado un self-adjoint operator $A$ sobre un espacio de Hilbert $\H$, decimos que un self-adjoint operator $\tilde{A}$ sobre
%$\H$ es una perturbaci\'on singular de $A$ si $D(A)\neq D(\tilde{A})$ y
%\begin{equation*}
%\textit{D}=\left\lbrace x\in D(A)\cap D(\tilde{A}): Ax=\tilde{A}x\right\rbrace 
%\begin{equation*}
%es denso en $\H$. Si el operador sim\'etrico cerrado densamente definido sobre $\H$, $\textbf{A}:=A\upharpoonright_{\textit{D}}=\tilde{A}\upharpoonright_{\textit{D}}$, tiene \'indices de deficiencia finitos,
%decimos que $\tilde{A}$ es una perturbaci\'on singular de rango finito de $A.$
%\end{definition}

%Bajo esta definici\'on, podemos decir que la familia de extensiones autoadjuntas de un operador sim\'etrico cerrado densamente definido sobre un espacio de Hilbert with \'indices de deficiencia $(n,n)$ dadas por una tripleta de frontera $\Phi$ son perturbaciones singulares de alg\'un operador perteneciente a la familia.

It is easy to prove the following identity.

\begin{lemma}\label{identidadweyl}
Let $B$ denote a densely defined closed symmetric operator on $\H$ with finite deficiency indices $(n,n)$ and
$\Pi=\left( \C^{n},\Gamma_{1},\Gamma_{2}\right)$ a boundary triplet for $B^{*}$. For $D,D'\in\Her_{n}$,
\begin{equation*}
M_{D}(z)=M_{D'}(z)\left[ (D-D')M_{D'}(z)+I\right]^{-1}= \left[ M_{D'}(z)(D-D')+I\right]^{-1}M_{D'}(z).
\end{equation*} 
\end{lemma}

The following is a result of \cite{MALAMUD} that generalizes both \cite[Theorem 4]{DON} as \cite[Theorem 4]{Aron}.
%It is worth mentioning that a proof of this theorem does not appear in said paper .

\begin{theorem}[Theorem 2, \cite{MALAMUD}]\label{malamud}
Suppose that $B$ is a symmetric completely non-selfadjoint operator on $\H$ with finite deficiency indices $(n,n)$.
Let $\Pi=\left( \C^{n},\Gamma_{1},\Gamma_{2}\right)$ be a boundary triplet for $B^{*}$ and $M$ the associated Weyl function. Let $\Omega$ be the matrix-valued measure associated with $M$.
Then, for every self-adjoint extension $A_{D}$ of $B$ with $D\in\Her_{n}$ 
\begin{equation*}
\sigma_{p}^{max}(A_{D})=\left\lbrace x\in\R:T(x)\in\C^{n\times n}\;and\;M(x+i0)=D\right\rbrace .
\end{equation*}
where $T:\R\longrightarrow\ \C^{n\times n}$ is given by
\begin{equation*}
T(x):=\int \frac{d\Omega(y)}{(x-y)^{2}}.
\end{equation*}
\end{theorem}

%In this work a proof is offered. For this purpose, we require the following theorem.
For its proof, we require the following theorem. 

\begin{theorem}[Corollary 7.10, \cite{MALAMUD2}]\label{corolario7.10deM}
Suppose hypothesis of Theorem \ref{malamud}. Let $A_{D}$ be a extension of $B$ with
$D\in\Her_{n}$, $M_{D}$ the Weyl function of $A_{D}$ and $\Omega_{D}$ the associated matrix-valued measure.
Let $Q_{\Omega_{D}}$ be the multiplication operator by $t$ on $L^{2}(\Omega_{D} ,\C^{n})$. Then $\H$ is unitary to $L^{2}\left(\Omega_{D},\C^{n}\right) $
and $A_{D}$ is unitarily equivalent to $Q_{\Omega_{D}}$.
%with domain
%\begin{equation}\label{operadordemult}
%D(Q_{\Omega}):=\left\lbrace f\in L^{2}(\Omega_{D} ,\C^{n}):Q_{\Omega}f\in L^{2}(\Omega_{D} ,\C^{n})\right\rbrace .
%\end{equation}
%and for $f\in D(Q_{\Omega})$, $Q_{\Omega}f(t):=tf(t)$. 
\end{theorem}

Due to the above we conclude the following.

\begin{prop}
Suppose the above hypothesis. Let $\E_{D}$ be the spectral family of $A_{D}$ with $D\in\Her_{n}$. Then
\begin{center}
$N_{\Omega_{D}}(t)=N_{A_{D}}(t)$, $\E_{D}$-a.e. on $\R$.
\end{center}
Furthermore, $N_{A_{D}}(\lambda)=dim\;Ker(A_{D}-\lambda I)$ if $\lambda$ is an eigenvalue of $A_{D}$.
\end{prop}
\begin{proof}
By the previous theorem, the operators  $A_{D}$ and $Q_{\Omega_{D}}$ are unitarily equivalent and the multiplicity
functions satisfy
\begin{center}
$N_{A_{D}}(t)=N_{Q_{\Omega_{D}}}(t)$, $\E_{D}$-a.e. on $\R$.
\end{center}
On the other hand, $\E_{Q_{\Omega_{D}}}$, $\E_{D}$, $\Omega_{D}$ and $Tr\Omega_{D}$ are equivalent and by
Theorem \ref{teoespectral}, $Q_{\Omega_{D}}$ is unitarily equivalent to the multiplication operator $Q$ on the direct integral
$\int_{\R}H(t)dTr\Omega_{D}(t)$. We know that $N_{Q_{\Omega_{D}}}(t)=dim\;H(t)$ and by Kats Theorem
$N_{\Omega_{D}}(t)=dim\;G(t)$.
Since $Q_{\Omega_{D}}$ is unitarily equivalent to $Q$, by Theorem \ref{teoespectralii}, 
\begin{center}
$N_{\Omega_{D}}(t)=dim\;G(t)=dim\;H(t)=N_{Q_{\Omega_{D}}}(t)$, $\E_{D}$-a.e. on $\R$.
\end{center}
Consequently, 
\begin{center}
$N_{\Omega_{D}}(t)=N_{A_{D}}(t)$, $\E_{D}$-a.e. on $\R$.
\end{center}
Now, we choose $f\in D(A_{D})$ such that $(A_{D}-\lambda)f=0$. By Theorem \ref{corolario7.10deM} and the analysis made 
we have the equality
\begin{center}
$(t-\lambda)f(t)=0$ where $f(t)\in G(t)$, $Tr\;\Omega_{D}$-a.e. on $\R$.
\end{center}
%Entonces $f(t)=0$ \'o $t=\lambda$, $Tr\;\Omega_{D}$-c.s. sobre $\R$. 
Then $f\in\int_{\R}\oplus G(t)dTr\Omega_{D}(t)$ satisfies the above equality if and only if $f(t)=0$ when $t\neq\lambda$ and
$f(\lambda)\neq 0$, $Tr\;\Omega_{D}$-a.e. on $\R$. Hence $Ker(A_{D}-\lambda I)$ is isomorphic to $G(\lambda)$.
\end{proof}

We can conclude the result that will serve us to prove Theorem \ref{malamud}.

\begin{corollary}\label{conjetura}
Suppose the last hypothesis. If $\lambda\in\sigma_{p}(A_{D})$, then
\begin{equation*}
rank\;\Omega_{D}(\left\lbrace \lambda\right\rbrace)=dim\;Ker(A_{D}-\lambda I).
\end{equation*} 
\end{corollary}
\begin{proof}
Let $f_{ij}$ be the entries of the density matrix of $\Psi_{\Omega_{D}}$. By Radon-Nikodym Theorem, 
\begin{center}
$\mu_{ij}(G)=\int_{G}f_{ij}(t)dTr\Omega_{D}(t)$, with $G\in\textit{B}(\R)$ bounded.
\end{center}
Therefore,
\begin{equation*}
\mu_{ij}(\left\lbrace \lambda\right\rbrace )=f_{ij}(\lambda)Tr\Omega_{D}(\left\lbrace \lambda\right\rbrace ).
\end{equation*}
Next,
\begin{equation*}
\Omega_{D}(\left\lbrace \lambda\right\rbrace)=\left( \mu_{ij}(\left\lbrace \lambda\right\rbrace)\right) _{i,j=1}^{n}
=\left( f_{ij}(\lambda)Tr\Omega_{D}(\left\lbrace \lambda\right\rbrace )\right) _{i,j=1}^{n}
= Tr\Omega_{D}(\left\lbrace \lambda\right\rbrace )\Psi_{\Omega_{D}}(\lambda).
\end{equation*}
Since $\lambda\in\sigma_{p}(A_{D})$, $\E_{D}\left( \left\lbrace \lambda\right\rbrace \right)\neq 0 $ and as $\E_{D}$,
$\Omega_{D}$ and $ Tr\Omega_{D}$ are equivalent, it turns out that $Tr\Omega_{D}(\left\lbrace \lambda\right\rbrace )\neq 0$. So,
\begin{equation*}
N_{\Omega_{D}}(\lambda)=rank\;\Omega_{D}(\left\lbrace \lambda\right\rbrace).
\end{equation*}
The conclusion follows by the last proposition and definition of multiplicity function of $\Omega_{D}$. 
%\begin{equation*}
%rank\;\Omega_{D}(\left\lbrace x\right\rbrace)=N_{A_{D}}(x)=dim\;Ker(A_{D}-xI).
%\end{equation*}
\end{proof}

With this, we follow the same strategy of the proof of \cite[Theorem 4]{Aron}.

\begin{proof}[Proof of Theorem \ref{malamud}]

Suppose $x\in\sigma_{p}^{max}(A_{D})$. By \cite[Theorem 5.5(i)]{GT},
\begin{equation}\label{masa}
\Omega_{D}(\left\lbrace x\right\rbrace )=-i\lim_{\epsilon\rightarrow 0}\epsilon M_{D}(x+i\epsilon).
\end{equation}

By Corollary \ref{conjetura}, $\Omega_{D}(\left\lbrace x\right\rbrace )^{-1}$ exists. Thus,
\begin{equation}\label{igualdaddelpeso}
\Omega_{D}(\left\lbrace x\right\rbrace)^{-1}
= i\lim_{\epsilon\rightarrow 0}\dfrac{1}{\epsilon} M_{D}(x+i\epsilon) ^{-1}
= i\lim_{\epsilon\rightarrow 0}\dfrac{1}{\epsilon} \left[ D-M(x+i\epsilon)\right] .
\end{equation}

Since $\Omega_{D}(\left\lbrace x\right\rbrace )^{-1}$ is a symmetric matrix,
\begin{align*}
\lim_{\epsilon\rightarrow 0}\dfrac{i}{\epsilon} \left[ D-M(x+i\epsilon)\right] 
&=\lim_{\epsilon\rightarrow 0}Re\left[ \dfrac{i}{\epsilon} \left( D-M(x+i\epsilon)\right)\right] \\
&=-\lim_{\epsilon\rightarrow 0}\dfrac{1}{\epsilon}Im\left[D-M(x+i\epsilon)\right]\\
&=-\lim_{\epsilon\rightarrow 0}\dfrac{-1}{\epsilon}Im\;M(x+i\epsilon).
\end{align*}
On the other hand 
\begin{equation*}
\int_{\R} \frac{d\Omega(y)}{(x-y)^{2}}=\lim_{\epsilon\rightarrow 0}\int_{\R} \frac{d\Omega(y)}{(x-y)^{2}+\epsilon^{2}}
=\lim_{\epsilon\rightarrow 0}\frac{1}{\epsilon} Im\;M(x+i\epsilon).
\end{equation*}
Therefore, 
\begin{equation*}
\int_{\R} \frac{d\Omega(y)}{(x-y)^{2}}=\Omega_{D}(\left\lbrace x\right\rbrace)^{-1}.
\end{equation*}
In addition, by (\ref{igualdaddelpeso}),
\begin{equation*}
\lim_{\epsilon\rightarrow 0}\left[ D-M(x+i\epsilon)\right] =
\lim_{\epsilon\rightarrow 0}\dfrac{1}{-i\epsilon} \left[ D-M(x+i\epsilon)\right] \lim_{\epsilon\rightarrow 0}(-i\epsilon)=0.
\end{equation*}
Thus, $M(x+i0)=D$.\\

Suppose that $\int_{\R} \frac{d\Omega(y)}{(x-y)^{2}}$ exists and $M(x+i0)=D$. For the representation (\ref{funcionM}),
\begin{align}\label{otraiguald}
M(x+i\epsilon)-D&=M(x+i\epsilon)-D+M(x+i0)-M(x+i0)\\
&= M(x+i\epsilon)-M(x+i0)\\
&=\int_{\R} \left( \frac{1}{y-x-i\epsilon}-\frac{1}{y-x}\right) d\Omega(y)\\
&=i\epsilon\int_{\R}\frac{d\Omega(y)}{(y-x-i\epsilon)(y-x)}.
\end{align}
It turns out that,
\begin{equation*}
T(x):=\int_{\R} \frac{d\Omega(y)}{(x-y)^{2}}=\lim_{\epsilon\rightarrow 0}\int_{\R}\frac{d\Omega(y)}{(y-x-i\epsilon)(y-x)}.
\end{equation*}
Since $T(x)$ is a positive-definite matrix, is invertible as well. By (\ref{masa}),
\begin{equation*}
\Omega_{D}(\left\lbrace x\right\rbrace )=\lim_{\epsilon\rightarrow 0}\left[ -i\epsilon M_{D}(x+i\epsilon)\right]
=\lim_{\epsilon\rightarrow 0}\left[\dfrac{M(x+i\epsilon)-D}{i\epsilon} \right] ^{-1}
=T(x)^{-1}.
\end{equation*}
Finally, by Corollary \ref{conjetura}, $x\in\sigma_{p}^{max}(A_{D})$.
\end{proof}

\section{Consequences of Malamud's Theorem}
\subsection*{A similar result}

We prove a result similar to Theorem \ref{malamud} in the sense that it changes the operator $A_{1}$ by another self-adjoint extension of $B$ disjoint of $A_{1}$.

\begin{theorem}
Suppose hypothesis of Theorem \ref{malamud} and $det(D-D^{'})\neq 0$ if $D,D'\in\Her_{n}$. Then
\begin{equation*}
\sigma_{p}^{max}(A_{D})=\left\lbrace x\in\R:\int_{\R} \frac{d\Omega_{D'}(y)}{(x-y)^{2}}\in\C^{n\times n}\;y\;M_{D'}(x+i0)=\left( D'-D\right)^{-1} \right\rbrace 
\end{equation*}
where $M_{D'}$ is the Weyl function of $A_{D'}$ and $\Omega_{D'}$ is the matrix-valued measure associated with $M_{D'}$.
\end{theorem}
\begin{proof}

Suppose that $x\in\sigma_{p}^{max}(A_{D})$. By (\ref{igualdaddelpeso}) and Lemma \ref{identidadweyl},

\begin{align*}
\Omega_{D}(\left\lbrace x\right\rbrace)^{-1}
&= i\lim_{\epsilon\rightarrow 0}\dfrac{1}{\epsilon} \left\lbrace M_{D'}(x+i\epsilon)\left[ (D-D')M_{D'}(x+i\epsilon)+I\right]^{-1} \right\rbrace ^{-1}\\
&=i\lim_{\epsilon\rightarrow 0}\dfrac{1}{\epsilon} \left[D-D'+M_{D'}(x+i\epsilon)^{-1}\right]
\end{align*}

Since $\Omega_{D}(\left\lbrace x\right\rbrace )^{-1}$ is a symmetric matrix, it turns out that
\begin{align*}
\Omega_{D}(\left\lbrace x\right\rbrace)^{-1} 
&=\lim_{\epsilon\rightarrow 0}Re\left\lbrace \dfrac{i}{\epsilon} \left[ D-D'+M_{D'}(x+i\epsilon)^{-1}\right]\right\rbrace  \\
&=-\lim_{\epsilon\rightarrow 0}\dfrac{1}{\epsilon}Im\left[D-D'+M_{D'}(x+i\epsilon)^{-1}\right]\\
&=-\lim_{\epsilon\rightarrow 0}\dfrac{1}{\epsilon}Im\;M_{D'}(x+i\epsilon)^{-1}\\
&=\lim_{\epsilon\rightarrow 0}\dfrac{1}{\epsilon}M_{D'}(x+i\epsilon)^{-1}Im\;M_{D'}(x+i\epsilon)M_{D'}(x-i\epsilon)^{-1}.
\end{align*}

By hypothesis and Theorem \ref{malamud},
\begin{equation}\label{valordefront}
\lim_{\epsilon\rightarrow 0}M_{D'}(x+i\epsilon)^{-1}=\lim_{\epsilon\rightarrow 0}\left[ D'-M(x+i\epsilon)\right]
=D'-M(x+i0)=D'-D.
\end{equation}
Analogously,
\begin{equation*}
\lim_{\epsilon\rightarrow 0}M_{D'}(x-i\epsilon)^{-1}=D'-D.
\end{equation*}
Therefore,
\begin{equation*}
\Omega_{D}(\left\lbrace x\right\rbrace)^{-1}=(D'-D)\lim_{\epsilon\rightarrow 0}\dfrac{Im\;M_{D'}(x+i\epsilon)}{\epsilon}(D'-D).
\end{equation*}

On the other hand, 
\begin{equation*}
\int_{\R} \frac{d\Omega_{D'}(y)}{(x-y)^{2}}=\lim_{\epsilon\rightarrow 0}\int_{\R} \frac{d\Omega_{D'}(y)}{(x-y)^{2}+\epsilon^{2}}
=\lim_{\epsilon\rightarrow 0}\frac{1}{\epsilon} Im\;M_{D'}(x+i\epsilon).
\end{equation*}

Thus, 
\begin{equation*}
\int_{\R} \frac{d\Omega_{D'}(y)}{(x-y)^{2}}=(D'-D)^{-1}\Omega_{D}(\left\lbrace x\right\rbrace)^{-1}(D'-D)^{-1}.
\end{equation*}
Moreover, by (\ref{valordefront}) it turns out that
\begin{equation*}
M_{D'}(x+i0)=\left[ \lim_{\epsilon\rightarrow 0}M_{D'}(x+i\epsilon)^{-1}\right]^{-1}=(D'-D)^{-1}.
\end{equation*}
So, $x$ meets the required conditions.\\

Suppose that $\int_{\R} \frac{d\Omega_{D'}(y)}{(x-y)^{2}}$ exists and $M_{D'}(x+i0)=(D'-D)^{-1}$. Analogous to (\ref{otraiguald})
\begin{equation*}
M_{D'}(x+i\epsilon)-(D'-D)^{-1}=i\epsilon\int_{\R}\frac{d\Omega_{D'}(y)}{(y-x-i\epsilon)(y-x)}.
\end{equation*}
%%%%%%%%%%%%%%%%%%%%%%%%%%%%%%%%%%%%%%%%%%%%%%%%%%%%%%%%%%%%%%
Hence,
\begin{equation*}
T_{D'}(x):=\int_{\R} \frac{d\Omega_{D'}(y)}{(x-y)^{2}}=\lim_{\epsilon\rightarrow 0}\int_{\R}\frac{d\Omega_{D'}(y)}{(y-x-i\epsilon)(y-x)}.
\end{equation*}

%%%%%%%%%%%%%%%%%%%%%%%%%%%%%%%%%%%%%%%%%%%%%%%%%%%%%%%%%%%%%%

Note that $T_{D'}(x)$ is a positive-definite matrix and therefore is invertible. By (\ref{masa}) and
Lemma \ref{identidadweyl},
\begin{align*}
\Omega_{D}(\left\lbrace x\right\rbrace )&=\lim_{\epsilon\rightarrow 0}\left[ -i\epsilon M_{D}(x+i\epsilon)\right]\\
&=\lim_{\epsilon\rightarrow 0}(-i\epsilon)M_{D'}(x+i\epsilon)\left[ (D-D')M_{D'}(x+i\epsilon)+I\right]^{-1}\\
&=\lim_{\epsilon\rightarrow 0}(-i\epsilon)M_{D'}(x+i\epsilon)\left\lbrace  (D'-D)\left[-M_{D'}(x+i\epsilon)+(D'-D)^{-1}\right] \right\rbrace ^{-1}\\
&=\lim_{\epsilon\rightarrow 0}M_{D'}(x+i\epsilon)\left[\frac{M_{D'}(x+i\epsilon)-(D'-D)^{-1}}{i\epsilon}\right]^{-1}(D'-D)^{-1}\\
&=(D'-D)^{-1}T_{D'}(x)^{-1}(D'-D)^{-1}.
\end{align*}
Thus, $\Omega_{D}(\left\lbrace x\right\rbrace )$ es invertible. Finally, by Corollary \ref{conjetura}, $x\in\sigma_{p}^{max}(A_{D})$.
\end{proof}

\subsection*{Forbidden Energies-like}

We conclude the following result.

\begin{theorem}
Suppose the hypothesis of Theorem \ref{malamud}. Then
\begin{equation}\label{setdeMalam}
\left\lbrace x\in\R:T(x)\in \C^{n\times n}\right\rbrace =
\bigcup\left\lbrace \sigma_{p}^{max}(A_{D}):D\in\Her_{n}\right\rbrace.
\end{equation}
where $T(x):=\int_{\R} \frac{d\Omega(y)}{(x-y)^{2}}$.
\end{theorem}
\begin{proof}

If $x\in\sigma_{p}^{max}(A_{D})$ for any $D\in\Her_{n}$, by Theorem \ref{malamud}
it turns out that $T(x)\in\C^{n\times n}$.\\

On the other hand, let $x\in\R$ such that $T(x)\in\C^{n\times n}$. Since $\Omega$ is the matrix-valued measure
associated with the Weyl function $M$, it turns out that the quadratic form  $\Phi_{f}$, where $f(y):=\frac{1}{1+y^{2}}$,
is defined everywhere in $\C^{n}$. This means that
\begin{equation*}
\int_{\R} \frac{d\Omega(y)}{1+y^{2}}\in\C^{n\times n}.
\end{equation*}
By Proposition \ref{implicaciongrl}, there is a $D\in\Her_{n}$ such that $M(x+i0)=D$. By Theorem \ref{malamud}, $x\in\sigma_{p}^{max}(A_{D})$. Therefore, the equality (\ref{setdeMalam}) is proven.
\end{proof}
%\begin{corollary}
%Suppose the hypothesis of Theorem \ref{malamud}. Then 
%\begin{equation}\label{u}
%\left\lbrace x\in\sigma(A_{1}):x\not\in\sigma_{p}^{max}(A_{D}),\forall\;D\in\Her_{n}\right\rbrace 
%\end{equation}
%
%is dense $G_{\delta}$ in $\sigma(A_{1})$.
%\end{corollary}
We finally conclude the main theorem.

\begin{proof}[Proof of Theorem \ref{maintheorem}]
We take complements in equality (\ref{setdeMalam}). So,
\begin{equation*}
\left\lbrace x\in\R:T(x)\not\in \C^{n\times n}\right\rbrace =
\left\lbrace x\in\R:x\not\in\sigma_{p}^{max}(A_{D}),\;for\;any\;D\in\Her_{n}\right\rbrace .
\end{equation*}
We intersect with $supp\;\Omega$ the sets of the previous equality. But we know that the spectral family $E_{A_{1}}$ of $A_{1}$ is equivalent to the matrix-valued measure $\Omega$. Therefore, $supp\;\Omega=supp E_{A_{1}}=\sigma(A_{1})$. Then the set (\ref{u}) is equal to the set (\ref{set}). Finally, by Theorem \ref{teo2.1gral}, we conclude the result.
%el theorem 3.5(iii) de \cite{MALAMUD2}
\end{proof}

\noindent\textbf{Acknowledgements}\\
\textit{This research was supported by CONACYT Grant No. 805144. The author thanks Prof. Luis Silva for helpful comments and especially Prof. Rafael del R\'io for his invaluable mentorship.}

\bigskip

\noindent Mario Ruiz\\
marioruiz@comunidad.unam.mx\\
ORCID: \url{https://orcid.org/0009-0009-0837-2996} \bigskip

\noindent {\small
\noindent Universidad Nacional Aut\'onoma de M\'exico\\
Instituto de Investigaciones en Matem\'aticas Aplicadas y en Sistemas\\
Departamento de F\'isica Matem\'atica\\
Ciudad de M\'exico, C.P. 04510.
}\bigskip

\end{document}